\documentclass[12pt]{amsart}

\usepackage{amssymb}
\usepackage{amsmath}
\usepackage{pstricks}
\usepackage{hyperref}
\newtheorem{theorem}{Theorem}
\newtheorem{lemma}[theorem]{Lemma}
\newtheorem{prop}[theorem]{Proposition}

\theoremstyle{definition}

\newcommand{\M}{\mathcal{M}}
\newcommand{\PM}{\mathcal{PM}}

\newcommand{\bdr}[1]{\partial\! #1}

\numberwithin{equation}{section}
\title[Generating the mapping class group]
{Generating the mapping class group of a
nonorientable surface by crosscap transpositions}
\author{Marta Le\`sniak}
\author{B{\l}a\.zej Szepietowski}
\address{Institute of Mathematics, Faculty of Mathematics, Physics and Informatics, University of Gda\'nsk, 80-308 Gda\'nsk, Poland} 
\email{marta.lesniak@mat.ug.edu.pl}
\email{blaszep@mat.ug.edu.pl}
\thanks{Both authors supported by National Science Centre, Poland, grant 2015/17/B/ST1/03235.}
\begin{document}
\begin{abstract}
A crosscap transposition is an element of the mapping class group of a nonorientable surface represented by a homeomorphism supported on a one-holed Klein bottle and swapping two crosscaps. We prove that the mapping class group of a compact nonorientable surface of genus $g\ge 7$ is generated by conjugates of one crosscap transposition. In the case when the surface is either closed or has one boundary component, we give an explicit set of $g+2$ crosscap transpositions  generating the mapping class group. 
\end{abstract}

\maketitle
\section{Introduction}
We say that a group $G$ is {\it normally generated} by a subset $X$  if $G$ is generated by conjugates of elements of $X$ and their inverses. In other words, $G$ is the normal closure of $X$ in $G$. For example, the symmetric group of a finite set is normally generated by a single transposition.

For a compact surface $F$, its {\it mapping class group} $\M(F)$ is the group of isotopy classes of all, orientation preserving if $F$ is orientable, homeomorphisms $F\to F$ equal to the identity on the boundary of $F$. A compact surface of genus $g$ with $n$ boundary components will be denoted by $S_{g,n}$ if it is orientable, or by $N_{g,n}$ if it is nonorientable.
One of the basic facts about $\M(S_{g,n})$, proved by Dehn \cite{Dehn}, is that it is generated by finitely many Dehn twists; see \cite{F-M} for a background on the mapping class group of an orientable surface. If $g\ge 2$ or $(g,n)\in\{(1,0),(1,1)\}$ then $\M(S_{g,n})$ is generated by Dehn twists about nonseparating curves, and because any two such twists are conjugate, $\M(S_{g,n})$ is normally generated by one twist. This fact, proved by Mumford \cite{Mum} and Lickorish \cite{Lick2}, has an important immediate consequence: the abelianisation of $\M(S_{g,n})$ is cyclic (and in fact trivial for $g\ge 3$; see \cite[Section 5.1]{F-M} for the history of the computation of the abelianisation of $\M(S_{g,n})$). 

\begin{figure}[t]
 \begin{center}
 \input{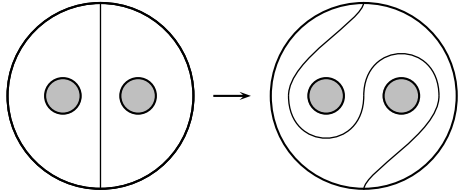}
\caption {A crosscap transposition.}\label{Fig:U}
\end{center}
\end{figure}
The purpose of this paper is to show that also $\M(N_{g,n})$ is normally generated by a single element, namely a crosscap transposition, if $g$ is at least $7$. A crosscap transposition is an element of $\M(N_{g,n})$ represented by a homeomorphism supported on a one-holed Klein bottle and swapping two crosscaps (see Figure \ref{Fig:U}, where the shaded discs represent crosscaps).
\begin{theorem}\label{T:main1}
For $g\ge 7$ and $n\ge 0$ the group $\M(N_{g,n})$ is normally generated by one crosscap transposition.
\end{theorem}
Since a conjugate of a crosscap transposition is also a crosscap transposition, $\M(N_{g,n})$ is generated by  crosscap transpositions for $g\ge 7$. In Theorem \ref{T:main2} we give an explicit set of $g+2$ crosscap transpositions generating $\M(N_{g,n})$ for $n\le 1$. The assumption $g\ge 7$  is necessary, as for $2\le g\le 6$ the abelianisation of $\M(N_{g,n})$ is not cyclic \cite{KorkH1,Stu_bdr}, and hence $\M(N_{g,n})$ can not be normally generated by one element.

Lickorish \cite{Lick1} proved that $\M(N_{g,0})$ is not generated by Dehn twists for $g\ge 2$, but it is generated by Dehn twists and Y-homeomorphisms, also called crosscap slides. The subgroup of $\M(N_{g,n})$ generated by all Dehn twists has index $2$ \cite{Lick3,Stu_twist}. Since crosscap slides induce the identity automorphism of $H_1(N_{g,n},\mathbb{Z}_2)$, they also generate a proper subgroup of 
 $\M(N_{g,n})$ (see \cite{Szep_GD}). 
 
Finite generating sets for $\M(N_{g,n})$  were found in \cite{Chill,KGD,Stu_twist,Stu_bdr} for $g\ge 3$ and arbitrary $n$, all consisting of Dehn twists and one crosscap slide. Because a crosscap transposition is the product of a crosscap slide and a Dehn twist, crosscap transpositions can be used instead of crosscap slides. The presentations of $\M(N_{g,n})$ given in \cite{PSz, Szep_Osaka, Szep_g4} use Dehn twists and crosscap transpositions as generators. Building on \cite{PSz}, Stukow \cite{Stu_Pres} found a finite presentation of $\M(N_{g,0})$ and $\M(N_{g,1})$ with Dehn twists and one crosscap slide as generators.   Crosscap transpositions appear already in the proof of the main  result of \cite{Chill}, and their name was introduced in \cite{Szep_g4}. Recently Parlak and Stukow \cite{Parlak} proved that in $\M(N_{g,0})$ crosscap slides have nontrivial roots if $g\ge 5$.

This paper is motivated by the desire to have a symmetric presentation of $\M(N_{g,n})$ in which all generators would have the same topological type (that is they would belong to the same conjugacy class). A good example of what we understand by a symmetric presentation is the finite presentation of $\M(S_{g,n})$ given by Gervais \cite{Gerv}, with many Dehn twists as generators but very simple relations.    We believe that the present paper is a step towards a similar presentation in the case of a nonorientable surface.

\section{Preliminaries}
Let $N=N_{g,n}$ be a nonoriantable surface of genus $g\ge 2$. We represent $N$ as a 2-sphere with $n$ holes and $g$ crosscaps (Figure \ref{Fig:Ngn}). In all figures of this paper shaded discs represent crosscaps. This means that the interior of each shaded disc should be removed, and then antipodal points on the resulting boundary component should be identified. 

By a {\it curve} in $N$ we understand in this paper an oriented  simple closed curve. We do not distinguish a curve from its isotopy class. If $\gamma$ is a curve, then $\gamma^{-1}$ is the curve with the same image as $\gamma$ but opposite orientation. According to whether a regular neighbourhood of $\gamma$ is an annulus or a M\"obius strip, we call $\gamma$ respectively two- or one-sided.

Suppose that $\mu$ and $\alpha$ are curves in $N$ such that $\mu$ is one-sided, $\alpha$ is two-sided and they intersect in one point. Let $K\subset N$ be a regular neighbourhood of $\mu\cup\alpha$, which is homeomorphic to the Klein bottle with a hole.
With the pair $\mu, \alpha$ we associate three elements of the mapping class group $\M(N)$: Dehn twists about $\alpha$, crosscap slide and crosscap transposition. In order to define a Dehn twist about $\alpha$ we need to choose an orientation of a regular neighbourhood of $\alpha$, and for that we use the second curve $\mu$. The orientations of $\mu$ and $\alpha$ determine a local orientation at their intersection point, which can be extended over a regular neighbourhood of $\alpha$.  We denote by  $T_{\mu,\alpha}$ the Dehn twist about $\alpha$ in the direction determined by $\mu$ as indicated by the small arrows on the left hand side of Figure \ref{Fig:YT}. By this convention we have
\[T_{\mu,\alpha}=T_{\mu^{-1},\alpha^{-1}}=T^{-1}_{\mu^{-1},\alpha}=T^{-1}_{\mu,\alpha^{-1}}.\]
\begin{figure}[t]
 \begin{center}
 \input{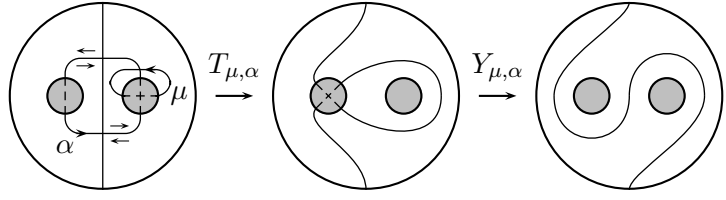}
\caption {The crosscap transposition $U_{\mu,\alpha}$ as the composition of the Dehn twist $T_{\mu,\alpha}$ and the crosscap slide $Y_{\mu,\alpha}$.}\label{Fig:YT}
\end{center}
\end{figure}
 If an orientation  of a regular neighbourhood of $\alpha$ is fixed, then we will usually use the standard notation $T_\alpha$ for a Dehn twist about  $\alpha$.  The {\it crosscap slide} $Y_{\mu,\alpha}$ is obtained by pushing a regular neighbourhood  of $\mu$ once along $\alpha$ keeping the boundary of $K$ fixed. Finally we define {\it crosscap transposition} to be the composition 
\[U_{\mu,\alpha}=Y_{\mu,\alpha}T_{\mu,\alpha}.\]
On Figure \ref{Fig:YT} we show the effect of $U_{\mu,\alpha}$ on an arc connecting two points of $\bdr{K}$ and separating two crosscaps.
Directly from the definitions, for $f\in\M(N)$ we have
\begin{align}
\label{eq:fTf}&fT_{\mu,\alpha}f^{-1}=T_{f(\mu),f(\alpha)}\\
\label{eq:fYf}&fY_{\mu,\alpha}f^{-1}=Y_{f(\mu),f(\alpha)}\\
\label{eq:fUf}&fU_{\mu,\alpha}f^{-1}=U_{f(\mu),f(\alpha)}.
\end{align}
Note that $Y_{\mu,\alpha}$ reverses the orientation of $\mu$ and preserves the orientation of $\alpha$. It  follows that $Y_{\mu,\alpha}$ reverses the orientation of a regular neighbourhood of $\alpha$, and by \eqref{eq:fTf} we have
\[Y_{\mu,\alpha}T_{\mu,\alpha}Y_{\mu,\alpha}^{-1}=T_{\mu^{-1},\alpha}=T^{-1}_{\mu,\alpha},\]
and also
\[U_{\mu,\alpha}T_{\mu,\alpha}U_{\mu,\alpha}^{-1}=T^{-1}_{\mu,\alpha}.\]
The inverse of $U_{\mu,\alpha}$ is $U_{\mu^{-1},\alpha^{-1}}$. Indeed, we have
\[U_{\mu^{-1},\alpha^{-1}}=Y_{\mu^{-1},\alpha^{-1}}T_{\mu^{-1},\alpha^{-1}}
=Y^{-1}_{\mu,\alpha}T_{\mu,\alpha}=T^{-1}_{\mu,\alpha}Y^{-1}_{\mu,\alpha}=U^{-1}_{\mu,\alpha}.\]
It follows from \eqref{eq:fUf} and the classification of surfaces that two crosscap transpositions $U_{\mu_1,\alpha_1}$ and $U_{\mu_2,\alpha_2}$ are conjugate in $\M(N)$ if and only if the complements of $\mu_1\cup\alpha_1$ and $\mu_2\cup\alpha_2$ in $N$ are either both nonorientable or both orientable.
\begin{figure}[t]
 \begin{center}
 \input{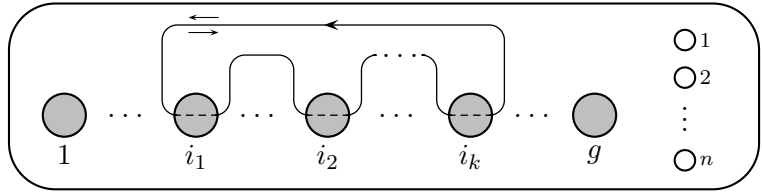}
\caption {The surface $N_{g,n}$ and the curve $\gamma_I$ for $I=\{i_1,i_2,\dots,i_k\}$.}\label{Fig:Ngn}
\end{center}
\end{figure}

For a non-empty subset $I=\{i_1,i_2,\dots,i_k\}\subseteq\{1,2,\dots,g\}$ let $\gamma_I$ denote the curve shown on Figure \ref{Fig:Ngn}. Note that $\gamma_I$ is two-sided if and only if $k$ is even. In such a case we denote by $T_{\gamma_I}$ the Dehn twists about $\gamma_I$ in the direction indicated by the small arrows on Figure \ref{Fig:Ngn} (that is $T_{\gamma_{\{i_k\}},\gamma_I}$ according to our convention).  The following curves will play a special role and so we give them different names:
\begin{itemize}
\item $\mu_{i}=\gamma_{\{i\}}$ for $i=1,2,\dots,g$,
\item $\alpha_{i}=\gamma_{\{i,i+1\}}$ for $i=1,2,\dots,g-1$,
\item $\beta=\gamma_{\{1,2,3,4\}}$.
\end{itemize}
We also define  Dehn twists
\begin{itemize}
\item $a_i=T_{\alpha_i}$ for $i=1,2,\dots,g-1$,
\item $b=T_{\beta}$,
\end{itemize}
and crosscap transpositions
\begin{itemize}
\item $u_i=U_{\mu_{i+1},\alpha_i}$ for $i=1,2,\dots,g-1$,
\item $v=U_{\mu_4,\beta}$.
\end{itemize}
Observe that $u_i$ swaps M\"obius bands with cores $\mu_i$ and $\mu_{i+1}$, while $v$ swaps M\"obius bands with cores $\mu_4$ and $\gamma_{\{1,2,3\}}$.
The following relations are consequences of \eqref{eq:fTf} and \eqref{eq:fUf}.
\begin{align}
\label{eq:uuu}& u_iu_{i+1}u_i=u_{i+1}u_iu_{i+1},\\
\label{eq:uua}& u_iu_{i+1}a_i=a_{i+1}u_iu_{i+1},
\end{align}
for $i=1,\dots,g-2$, and
\begin{align}
\label{eq:uvu}& u_4vu_4=vu_4v,\\
\label{eq:uva}& u_4va_4=bu_4v.
\end{align}
Indeed,  \eqref{eq:uuu} and \eqref{eq:uua} follow from the fact that $u_iu_{i+1}$ takes $\alpha_i$ to $\alpha_{i+1}$ and $\mu_{i+1}$ to $\mu_{i+2}$. Similarly, \eqref{eq:uvu} and \eqref{eq:uva} follow from the fact that $u_4v$ takes $\alpha_4$ to $\beta$ and $\mu_{5}$ to $\mu_{4}$.

\section{Proofs}
\begin{prop}\label{P:au}
For $g\ge 5$ and $n\ge 0$, $\M(N_{g,n})$ is normally generated by $\{a_1,u_1\}$.
\end{prop}
\begin{proof}
Set $N=N_{g,n}$ and let $\widehat{N}$ be the surface obtained by gluing a once-punctured disc to each boundary component of $N$. Let $\PM^+(\widehat{N})$ denote the group of isotopy classes of homeomorphisms of $\widehat{N}$ fixing each puncture and preserving local orientation at each puncture.  Let $d_1,\dots d_n\in\M(N)$ be  Dehn twists about curves parallel to the boundary components of $N$.
 The inclusion of $N$ in $\widehat{N}$ induces a surjective  homomorphism $\M(N)\to\PM^+(\widehat{N})$ whose kernel is generated by $d_1,\dots,d_n$.
 By \cite[Proposition 2.10]{Ata_Szep} $\PM^+(\widehat{N})$ is generated by $u_1$ and Dehn twists about nonseparating curves with nonorientable complement. It follows that $\M(N)$ is normally generated by $\{u_1, a_1, d_1,\dots,d_n\}$. It suffices to show that $d_i$ can be expressed as a product of conjugates of $a_1$ for $i=1,\dots,n$. Let $S$ be an embedded orientable subsurface of $N$ of genus $2$ and such that $\bdr{N}\subset\bdr{S}$. By \cite{LabPar} $\M(S)$ is generated by Dehn twists about nonseparating curves. It follows that each $d_i$ is a product of Dehn twists about nonseparting curves in $S$. Because the complement in $N$ of such a curve is  nonorientable, the associated Dehn twist is conjugate to $a_1$. 
\end{proof}
\begin{lemma}\label{L:a1}
For $g\ge 7$ and $n\ge 0$ the twist $a_1$ is in the subgroup of $\M(N_{g,n})$ generated by the following set of crosscap transpositions 
\[\{u_i, v, a_3u_2a_3^{-1}, a_3u_3a_3^{-1}\,\mid\,1\le i\le 6\}.\]
\end{lemma}
\begin{proof} Let $G$ be the subgroup of $\M(N_{g,n})$ generated by the set given in the lemma. We begin by showing that $a_3u_ia_3^{-1}\in G$ for $1\le i\le 6$. This follows from the definition of $G$ for $i=2,3$ and from the fact that $a_3$ commutes with $u_i$ for $i=1,5,6$. To prove $a_3u_4a_3^{-1}\in G$ we use the fact that $u_4$ commutes with $a_2$ and $a_2=(u_2u_3)^{-1}a_3(u_2u_3)$ by \eqref{eq:uua}. We have
\begin{align*}
&u_4=a_2u_4a_2^{-1}=(u_2u_3)^{-1}a_3(u_2u_3)u_4(u_2u_3)^{-1}a_3^{-1}(u_2u_3),\\
&a_3u_4a_3^{-1}=a_3(u_2u_3)^{-1}a_3^{-1}(u_2u_3)u_4(u_2u_3)^{-1}a_3(u_2u_3)a_3^{-1}\in G.
\end{align*}
Analogously we obtain $a_3va_3^{-1}\in G$ from the fact that $v$ also commutes with $a_2$ (the curves $\beta$ and $\mu_4$ used to define $v$ are disjoint from $\alpha_2$).
\begin{figure}[t]
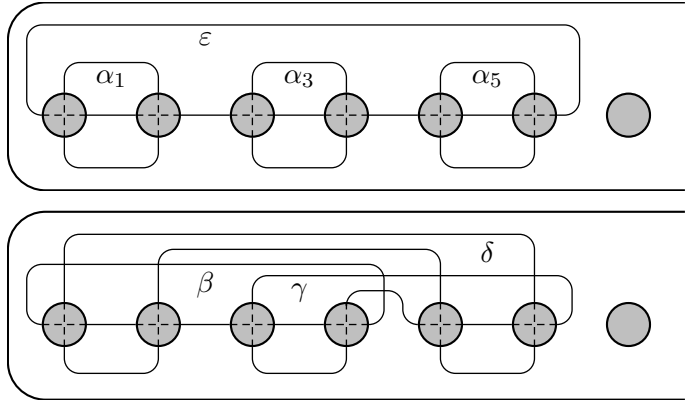

 \begin{center}
 \pspicture*(10,2.5)
%
%
\psline[linearc=.5](9,2.5)(0,2.5)(0,0)(9,0)
\pscircle*[linecolor=lightgray](.75,1){.3}
\pscircle(.75,1){.3}
\psline[linewidth=.5pt,linestyle=dashed,dash=3pt 2pt](.45,1)(1.05,1)
\psline[linewidth=.5pt,linestyle=dashed,dash=3pt 2pt](.75,.7)(.75,1.3)
\psline[linewidth=.5pt](1.05,1)(1.7,1)
\pscircle*[linecolor=lightgray](2,1){.3}
\pscircle(2,1){.3}
\psline[linewidth=.5pt,linestyle=dashed,dash=3pt 2pt](1.7,1)(2.3,1)
\psline[linewidth=.5pt,linestyle=dashed,dash=3pt 2pt](2,.7)(2,1.3)
\psline[linewidth=.5pt,linearc=.2](.75,1.3)(.75,1.7)(2,1.7)(2,1.3)
\psline[linewidth=.5pt,linearc=.2](.75,.7)(.75,.3)(2,.3)(2,.7)
\rput[t](1.375,1.6){\small$\alpha_1$}
\psline[linewidth=.5pt](2.3,1)(2.95,1)
\pscircle*[linecolor=lightgray](3.25,1){.3}
\pscircle(3.25,1){.3}
\psline[linewidth=.5pt,linestyle=dashed,dash=3pt 2pt](2.95,1)(3.55,1)
\psline[linewidth=.5pt,linestyle=dashed,dash=3pt 2pt](3.25,.7)(3.25,1.3)
\psline[linewidth=.5pt](3.55,1)(4.2,1)
\pscircle*[linecolor=lightgray](4.5,1){.3}
\pscircle(4.5,1){.3}
\psline[linewidth=.5pt,linestyle=dashed,dash=3pt 2pt](4.2,1)(4.8,1)
\psline[linewidth=.5pt,linestyle=dashed,dash=3pt 2pt](4.5,.7)(4.5,1.3)
\psline[linewidth=.5pt,linearc=.2](3.25,1.3)(3.25,1.7)(4.5,1.7)(4.5,1.3)
\psline[linewidth=.5pt,linearc=.2](3.25,.7)(3.25,.3)(4.5,.3)(4.5,.7)
\rput[t](3.875,1.6){\small$\alpha_3$}
\psline[linewidth=.5pt](4.8,1)(5.45,1)

\pscircle*[linecolor=lightgray](5.75,1){.3}
\pscircle(5.75,1){.3}
\psline[linewidth=.5pt,linestyle=dashed,dash=3pt 2pt](5.45,1)(6.05,1)
\psline[linewidth=.5pt,linestyle=dashed,dash=3pt 2pt](5.75,.7)(5.75,1.3)
\psline[linewidth=.5pt](6.05,1)(6.7,1)
\pscircle*[linecolor=lightgray](7,1){.3}
\pscircle(7,1){.3}
\psline[linewidth=.5pt,linestyle=dashed,dash=3pt 2pt](6.7,1)(7.3,1)
\psline[linewidth=.5pt,linestyle=dashed,dash=3pt 2pt](7,.7)(7,1.3)
\psline[linewidth=.5pt,linearc=.2](5.75,1.3)(5.75,1.7)(7,1.7)(7,1.3)
\psline[linewidth=.5pt,linearc=.2](5.75,.7)(5.75,.3)(7,.3)(7,.7)
\rput[t](6.375,1.6){\small$\alpha_5$}
\pscircle*[linecolor=lightgray](8.25,1){.3}
\pscircle(8.25,1){.3}
\psline[linewidth=.5pt,linearc=.2](.45,1)(.25,1)(.25,2.2)(7.6,2.2)(7.6,1)(7.3,1)
\rput[t](2.625,2.1){\small$\varepsilon$}
%
%
%
%
%
\endpspicture\\
 \pspicture*(10,2.75)
%
%
\psline[linearc=.5](9,2.5)(0,2.5)(0,0)(9,0)
\pscircle*[linecolor=lightgray](.75,1){.3}
\pscircle(.75,1){.3}
\psline[linewidth=.5pt,linestyle=dashed,dash=3pt 2pt](.45,1)(1.05,1)
\psline[linewidth=.5pt,linestyle=dashed,dash=3pt 2pt](.75,.7)(.75,1.3)
\psline[linewidth=.5pt](1.05,1)(1.7,1)
\pscircle*[linecolor=lightgray](2,1){.3}
\pscircle(2,1){.3}
\psline[linewidth=.5pt,linestyle=dashed,dash=3pt 2pt](1.7,1)(2.3,1)
\psline[linewidth=.5pt,linestyle=dashed,dash=3pt 2pt](2,.7)(2,1.3)
\psline[linewidth=.5pt](2.3,1)(2.95,1)
\pscircle*[linecolor=lightgray](3.25,1){.3}
\pscircle(3.25,1){.3}
\psline[linewidth=.5pt,linestyle=dashed,dash=3pt 2pt](2.95,1)(3.55,1)
\psline[linewidth=.5pt,linestyle=dashed,dash=3pt 2pt](3.25,.7)(3.25,1.3)
\psline[linewidth=.5pt](3.55,1)(4.2,1)
\pscircle*[linecolor=lightgray](4.5,1){.3}
\pscircle(4.5,1){.3}
\psline[linewidth=.5pt,linestyle=dashed,dash=3pt 2pt](4.2,1)(4.8,1)
\psline[linewidth=.5pt,linestyle=dashed,dash=3pt 2pt](4.5,.7)(4.5,1.3)
%
\psline[linewidth=.5pt,linearc=.2](3.25,1.3)(3.25,1.65)(7.5,1.65)(7.5,1)(7.3,1)
\psline[linewidth=.5pt,linearc=.2](3.25,.7)(3.25,.35)(4.5,.35)(4.5,.7)
\psline[linewidth=.5pt,linearc=.2](4.5,1.3)(4.5,1.45)(5.25,1.45)(5.25,1)(5.45,1)
\rput[t](3.875,1.55){\small$\gamma$}
\pscircle*[linecolor=lightgray](5.75,1){.3}
\pscircle(5.75,1){.3}
\psline[linewidth=.5pt,linestyle=dashed,dash=3pt 2pt](5.45,1)(6.05,1)
\psline[linewidth=.5pt,linestyle=dashed,dash=3pt 2pt](5.75,.7)(5.75,1.3)
\psline[linewidth=.5pt](6.05,1)(6.7,1)
\pscircle*[linecolor=lightgray](7,1){.3}
\pscircle(7,1){.3}
\psline[linewidth=.5pt,linestyle=dashed,dash=3pt 2pt](6.7,1)(7.3,1)
\psline[linewidth=.5pt,linestyle=dashed,dash=3pt 2pt](7,.7)(7,1.3)
\pscircle*[linecolor=lightgray](8.25,1){.3}
\pscircle(8.25,1){.3}
%

%
%
\psline[linewidth=.5pt,linearc=.2](.45,1)(.25,1)(.25,1.8)(5,1.8)(5,1)(4.8,1)
\rput[t](2.625,1.7){\small$\beta$}
%

\psline[linewidth=.5pt,linearc=.2](.75,.7)(.75,.35)(2,.35)(2,.7)
\psline[linewidth=.5pt,linearc=.2](5.75,.7)(5.75,.35)(7,.35)(7,.7)
\psline[linewidth=.5pt,linearc=.2](.75,1.3)(.75,2.2)(7,2.2)(7,1.3)
\psline[linewidth=.5pt,linearc=.2](2,1.3)(2,2)(5.75,2)(5.75,1.3)
\rput[t](6.375,2.1){\small$\delta$}
%
%
%
%
%
%
\endpspicture
\caption {The curves of the lantern relation.}\label{Fig:Lantern}
\end{center}
\end{figure}
Let us consider the curves $\alpha_{1}$, $\alpha_{3}$, $\alpha_{5}$, $\beta$ as defined above and the curves $\gamma$, $\delta$, $\varepsilon$ defined as following:
\[\gamma=\gamma_{\{3,4,5,6\}},\quad \delta=\gamma_{\{1,2,5,6\}},\quad \varepsilon=\gamma_{\{1,2,3,4,5,6\}}.\]
We also define Dehn twists:
$c=T_{\gamma}$, $d=T_{\delta}$, $e=T_{\varepsilon}$.
Observe that the curves $\alpha_1$, $\alpha_3$, $\alpha_5$ and $\varepsilon$ bound a 4-holed sphere (Fig. \ref{Fig:Lantern}) and the twists $a_{1}, a_{3}, a_{5}, b, c, d, e$ satisfy the lantern relation (see \cite[Section 5.1]{F-M}):
$$
a_{1}a_{3}a_{5}e=bcd.
$$
From that we obtain $a_{1}$ as follows:
$$
a_{1}=(ba_{3}^{-1})(ca_{5}^{-1})(de^{-1}).
$$
We will show that each of the factors $(ba_{3}^{-1})$, $(ca_{5}^{-1})$, $(de^{-1})$ is in $G$. By \eqref{eq:uva} and \eqref{eq:uua} we have $b=(u_4v)a_4(u_4v)^{-1}$ and $a_4=(u_3u_4)a_3(u_3u_4)^{-1}$.
Thus 
\[
ba_3^{-1}=(u_4vu_3u_4)a_3(u_4vu_3u_4)^{-1}a_3^{-1}\in G.
\]
We set
\[x=u_2u_3u_4u_5u_1u_2u_3u_4.\]
Observe that $x(\beta)=\gamma$, $x(\alpha_3)=\alpha_5$ and $x(\mu_4)=\mu_6$. Thus \[ca_5^{-1}=x(ba_3^{-1})x^{-1}\in G.\]
We define the crosscap transposition $w$ as follows:
\[w=xvx^{-1}=U_{\mu_{6},\gamma}.\]
We have $w=Y_{\mu_{6},\gamma}T_{\mu_{6},\gamma}$ and it is easy to check that \[w(\gamma_{\{1,2,6,7\}})=\gamma_{\{1,2,3,4,5,7\}}.\] Note that $u_5u_6(\gamma_{\{1,2,6,7\}})=\delta$ and $u_6(\gamma_{\{1,2,3,4,5,7\}})=\varepsilon$,  and finally $u_{6}wu_{5}^{-1}u_{6}^{-1}(\delta)=\varepsilon$. Thus
\begin{align*}
& e=(u_{6}wu_{5}^{-1}u_{6}^{-1})d(u_{6}wu_{5}^{-1}u_{6}^{-1})^{-1}\\
& de^{-1}=d(u_{6}wu_{5}^{-1}u_{6}^{-1})d^{-1}(u_{6}wu_{5}^{-1}u_{6}^{-1})^{-1}.
\end{align*}
Since $u_{6}wu_{5}^{-1}u_{6}^{-1}$ is a product of $u_i$ and $v$, to prove $de^{-1}\in G$ it suffices to show that $du_id^{-1}\in G$ and $dvd^{-1}\in G$ for $1\le i\le 6$. This follows from the fact that  $d$  can be obtained by conjugating $a_3$ by some 
 product of $u_i$ and $v$. Indeed, we have
 $u_4u_5u_3u_4(\delta)=\beta$ and
 \begin{align*}
 d &=(u_4u_5u_3u_4)^{-1}b(u_4u_5u_3u_4)=\\
&=(u_4u_5u_3u_4)^{-1}(u_4vu_3u_4)a_3(u_4vu_3u_4)^{-1}(u_4u_5u_3u_4).\qedhere
 \end{align*}
\end{proof}
We are now in a position to prove our main theorem.
\begin{proof}[Proof of Theorem \ref{T:main1}]
By Proposition \ref{P:au} $\M(N_{g,n})$ is normally generated by $u_1$ and $a_1$. By Lemma \ref{L:a1} $a_1$ can be expressed as a product of conjugates of $u_1$, and hence $\M(N_{g,n})$ is normally generated by $u_1$.
\end{proof}
\begin{theorem}\label{T:main2}
For $g\ge 7$ and $n\in\{0,1\}$ $\M(N_{g,n})$ is generated by the following set of crosscap transpositions
\[\{u_i, v, a_3u_2a_3^{-1}, a_3u_3a_3^{-1}\,\mid\,1\le i\le g-1\}.\]
\end{theorem}
\begin{proof} Let $N=N_{g,n}$, where $g\ge 7$ and $n\in\{0,1\}$. 
It follows from the presentation of $\M(N)$ given in \cite{PSz} that this group is generated by $a_i, u_i$ for $i=1,\dots,g-1$ and $b$.
By \eqref{eq:uva} we have $b=(u_4v)a_4(u_4v)^{-1}$ and by \eqref{eq:uua} every $a_k$ can be obtained by conjugating $a_1$ by some product of $u_i$. Therefore we see that $\M(N)$ is generated by  $a_1$, $v$ and $u_i$ for $i=1,\dots,g-1$. Theorem \ref{T:main2} follows by 
 Lemma \ref{L:a1}.
\end{proof}

\end{document}